\documentclass[11pt]{article}

\usepackage{amsmath}
\usepackage{amsthm}

\usepackage{amssymb}

\title{The extremal symmetry of arithmetic simplicial complexes}
\author{Benson Farb and Amir Mohammadi \thanks{BF is
supported in part by the NSF.}}

\theoremstyle{plain}
\newtheorem{theorem}{Theorem}[section]

\newtheorem{proposition}[theorem]{Proposition}
\newtheorem{fact}[theorem]{Fact}
\newtheorem{lemma}[theorem]{Lemma}

\def\title{\em}

\def\bar{\overline}

\newcommand\R{\mbox{\bf R}}

\newcommand\F{\mbox{\bf F}}

\DeclareMathOperator\Fix{Fix}

\DeclareMathOperator\Out{Out}
\DeclareMathOperator\Aut{Aut}
\DeclareMathOperator\rank{rank}

\DeclareMathOperator\Isom{Isom}

\DeclareMathOperator\chr{char}

\def\Dbf{\sigma}
\def\bfg{\mathbf{G}}
\def\bfz{\mathbf{Z}}
\def\zbf{\mathbf{Z}}
\def\qbf{\mathbf{Q}}

\def\tbg{\widetilde{\bfg}}
\def\tg{\widetilde{G}}
\def\fbf{\mathbf{F}}

\def\bg{B_\Gamma}
\def\cg{C}

\def\PGL{{\rm{PGL}}}

\def\Aut{{\rm{Aut}}}
\def\id{\operatorname{id}}
\def\h{\hspace{1mm}}

\begin{document}
\maketitle

\section{Introduction}\label{sec;intro}

Let $K$ be a nonarchimedean local field, for example the $p$-adic numbers $\mathbf{Q}_p$ (${\rm char}(K)=0$) or the field of Laurent series over a finite field $\F_p((t))$  (${\rm char}(p)>0$) .  
Let $G=\PGL_n(K)$, or more generally the $K$-points of any absolutely simple, connected, algebraic $K$-group of adjoint form.  

There is a natural way to associate to each 
cocompact lattice $\Gamma$ in $G$ a finite simplicial complex $B_\Gamma$, as follows.  Bruhat-Tits theory (see below) provides a contractible, $\rank_KG$-dimensional simplicial 
complex $X_G$ on which $G$ acts by simplicial automorphisms. The lattice
$\Gamma$ acts properly discontinuously on $X_G$ with
quotient a simplicial complex $B_\Gamma$. \footnote{If
$\Gamma$ has torsion, one needs to barycentrically
subdivide each simplex in $X_G$ in order to make the
quotient a true (not orbi) simplicial complex.}  

Margulis proved (see, e.g., \cite{Ma}) 
that $\rank_KG\geq 2$ implies that every lattice $\Gamma$ in $G$ is arithmetic.   We also note that $\chr(K)=0$ implies every lattice in $G(K)$ is cocompact.  In this paper we explore one aspect of the theme that, since the complex $B_\Gamma$ is constructed using number theory, it should have remarkable properties.  Here we concentrate on the extremal nature of the symmetry of $B_\Gamma$ and all of its covers.

Our first result shows that the simplicial structure of  $B_\Gamma$ realizes all simplicial symmetries of any simplicial complex homeomorphic to $B_\Gamma$.  For any simplicial complex $C$ we denote by $\Aut(C)$ the group of simplicial automorphisms of $C$.  We denote by $|C|$ the simplicial complex $C$ thought of as a topological space, without remembering the simplicial structure.

\begin{theorem}
\label{theorem:first}
Let $K$ be a nonarchimedean local field, and let $G$ be the $K$-points of an absolutely simple, connected algebraic $K$-group of adjoint form with $\rank_KG\geq 2$.  Let $\Gamma$ be a cocompact lattice in $G$, and let 
$B_\Gamma$ be the quotient by $\Gamma$ 
of the Bruhat-Tits building associated to $G$.  Suppose $C$ is any 
simplicial complex homeomorphic to $|B_\Gamma|$.  Then there is an injective homomorphism
$$\Aut(C)\longrightarrow \Aut(B_\Gamma).$$
\end{theorem}

Of course the simplicial structure on the space $|B_\Gamma|$ coming from the Bruhat-Tits building is not the unique simplicial structure satisfying Theorem \ref{theorem:first}.  One can, for example, take all the top-dimensional simplices of $B_\Gamma$ and subdivide them in the same way, so that the triangulation restricted to any maximal simplex gives a fixed simplicial isomorphism type.  Each of these new triangulations of $|B_\Gamma|$ has automorphism group $\Aut(B_\Gamma)$.   We call such a simplicial structure on $B_\Gamma$ an {\em arithmetic simplicial structure}.  

Our main result is a rigidity theorem characterizing arithmetic simplicial structures among all simplicial structures on $|B_\Gamma|$.  It gives 
a universal constraint on the symmetry of the universal covers of all other simplicial structures on $|B_\Gamma|$.

\begin{theorem}
\label{theorem:second}
Let $G$ and $\Gamma$ as in Theorem \ref{theorem:first} be given.   Fix a normalization of Haar measure $\mu$ on $G$.  Then there exists a constant $N\geq 1$, depending only on the 
$\mu(G/\Gamma)$,  with the following property: Let $C$ be any simplicial complex homeomorphic to $|B_\Gamma|$, and let $Y$ be the universal cover of $C$ (which therefore inherits a $\Gamma$-equivariant simplicial structure from $C$).
Then either:
\begin{enumerate}
\item $[\Aut(Y):\Gamma]<N$, so in particular $\Aut(Y)$ is finitely generated, or
\item $C$ is an arithmetic simplicial structure, and so $\Aut(Y)$ is uncountable and acts transitively on chambers (simplecis of maximal dimension.)
\end{enumerate}

\end{theorem}

\begin{remarks}
\hfill
\begin{enumerate}
\item Theorem \ref{theorem:second} is not true in the case that $\rank_KG=1$, i.e.\ when 
$X_G$ is a tree. An example is given in Section~\ref{sec;example}. The obstruction in this case is the fact that $\Aut(X_G)$ is ``far" from $G.$ 

\item One is tempted to weaken the hypotheses of Theorem \ref{theorem:first} and Theorem \ref{theorem:second}, for example to only require that $C$ is homotopy equivalent to $|B_\Gamma|$ rather than homeomorphic to it.  However the conclusion of each theorem is not true in this case, 
even for $C$ of the same dimension as $|B_\Gamma|$.  One can see this by taking, for any given $n\geq 2$, a triangulation of the closed disk $D^2$ by dividing $D^2$ into $n$ equal sectors based at the origin. This triangulation is   
invariant by the $2\pi/n$ rotation. Now let $C$ be the complex obtained by attaching the central vertex of $D^2$ to some vertex of $B_\Gamma$.  It is clear that $\Aut(C)$ contains $\bfz/n\bfz$.  
Since $n\geq 2$ was arbitrary,  he conclusions both of Theorem \ref{theorem:first} and of Theorem \ref{theorem:second} do not hold.

\item Theorem \ref{theorem:first} (resp. Theorem \ref{theorem:second}) is a simplicial analogue of a theorem of Farb-Weinberger from Riemannian geometry, given in \cite{FW1} (resp. \cite{FW2}).  However, the mechanism giving rigidity is different here.  Further, the type of generality achieved in the theorems in \cite{FW2} seems not to be possible in the simplicial setting, since  counterexamples abound, as the last remark indicates.
\end{enumerate}
\end{remarks}

One consequence of Theorem \ref{theorem:second} is the following.  Suppose $B_\Gamma$ has more than one top-dimensional simplex; this can always be achieved by passing to a finite index subgroup of $\Gamma$.  Now build a new triangulation $C$ of $|B_\Gamma|$ by subdividing the top-dimensional simplices of $B_\Gamma$, so that the resulting triangulations on some pair 
of such simplices are not simplicially isomorphic.  Then Theorem \ref{theorem:second} implies that $[\Aut(Y):\Gamma]<\infty$.  

Another way to think of this is that, if we paint the (open) top-dimensional simplices of $B_\Gamma$ with colors, and if we use at least $2$ distinct colors, the group of color-preserving automorphisms of the universal cover of $B_\Gamma$ is discrete, and contains $\Gamma$ as a subgroup of finite index.  This result is actually an ingredient in the proof of Theorem \ref{theorem:second}, and so is proven first. Such a result does not hold when $\rank_KG=1$. We give explicit examples of this failure in Section~\ref{sec;example}. 

\bigskip
\noindent
{\bf Outline of paper. }After giving some preliminary material on Euclidean buildings in \S\ref{sec;notation}, we prove the main results in \S\ref{sec;proof}.  In \S\ref{sec;example} we 
give an explicit example of a $B_\Gamma$ satisfying the hypotheses of Theorem \ref{theorem:first} and Theorem \ref{theorem:second} and a non-example in rank one case.

\bigskip
\noindent
{\bf Standing assumption. }All simplicial structures considered in this paper are assumed to be locally finite.

{\bf Acknowledgments.}
We would like to thank K.~Wortman and S. Weinberger for many helpful comments.

\section{Geometry and automorphisms of Euclidean buildings}
\label{sec;notation}

We now recall some facts from Bruhat-Tits theory which will be needed in this paper.  We refer the reader to \cite{AB},~\cite{W} and to \cite{Ti2} for these facts and definitions of terms.

\subsection{The building ${\bf X_G}$ }

Let $K$ be a nonarchimedean local field. Let $\bfg$ be the adjoint form of an absolutely almost simple, connected,  simply connected algebraic group defined over $K.$ Let $G=\bfg(K)$.

The Bruhat-Tits theory associates a contractible simplicial complex $X_G$ to $G$ on which $G$ acts by simplicial automorphisms.  This is easiest to describe if we work with the simply connected cover of $\bfg.$ So let $\tbg$ be the simply connected cover of $\bfg$ and let $\tg=\tbg(K).$   
Let $$r:=\rank_K(\bfg)$$

An {\em Iwahori subgroup} $I$ of $\tg$ is the normalizer of a Sylow pro-$p$-subgroup of $\tg.$ These subgroups are conjugate to each other since the Sylow subgroups are conjugate. The {\em Euclidean (or affine)  building} $X_G$ associated with $G$ is a simplicial complex defined as follows.  The 
vertices of $X_G$ correspond bijectively with maximal compact subgroups of $\tg$. A collection of maximal compact subgroups gives a simplex in $X_G$ precisely when their intersection contains an Iwahori subgroup. $X_G$ is a contractible simplicial complex whose dimension equals ${\rm{rank}}_K\bfg.$ In particular, if ${\rm{rank}}_K\bfg=1$ then $X_G$ is a tree.

We will need the following properties of $X_G$.

\begin{enumerate}
\item $X_G$ is {\em thick}; that is, any $i$-simplex of $X_G$ with $i<r:=\dim(X_G)$ is contained in at least three $(i+1)$-simplices.  

\item  Given any apartment (maximal flat) $A$ in $X_G$, any 
$(r-1)$-dimensional simplex lying in $A$ is contained in precisely two $r$-simplecis of $A$.

\item Any two simplices of $X_G$ are contained in a common apartment.
\end{enumerate}

\subsection{The action of ${\bf G}$}

The groups $\tg$ and $G$ act simplicially on $X_G$ by conjugation.  The stabilizer in $\tg$ of any vertex of $X_G$ is a maximal compact subgroup of $\tg.$   
There are $r+1$ orbits of vertices of $X_G$ under the $\tg$-action.  In this way each vertex is given a {\em type}. The action of $\tg$ on $X_G$ is type-preserving, and is transitive on the set of chambers (simplecies of maximal dimension) in $X_G$.   

Let $G^+$ be the normal subgroup of $G$ generated by all the unipotent radicals of $K$-parabolic subgroups of $G.$ The group $G^+$ is the image of $\tg$ under the covering map 
$\tbg\rightarrow \bfg$.  For example, if $G=\PGL_n(K)$ then $G^+=\mbox{PSL}_n(K)$; see e.g.~\cite[Chapter I]{Ma}. The covering map restricted to the unipotent subgroups is injective since the kernel of the covering map is the center of $\tbg$.  The subgroup $G^+$ is cocompact in $G$, and indeed is finite index when ${\rm char}(K)=0$.  
Further, $G^+$ acts by type-preserving automorphisms on $X_G$.

Denote by $\Aut_{\rm alg}(G)$ the group of algebraic automorphisms of $G$.  This group is the semidirect product of $G$ with the group of automorphisms of the Dynkin diagram for (the Lie algebra corresponding to) $G$, which is a group of order at most $2$ 
(see~\cite[Theorem 2.8]{PR}).  Let $\Aut_G(K)$ denote the group of field automorphisms $\sigma$ of $K$ such that ${}^\sigma \bfg$ and $\bfg$ ar $K$-isomorphic where ${}^\sigma \bfg$ is the group obtained from $\bfg$ by applying $\sigma$ to the defining equations.  The group $G$ is a locally compact topological group under the topology coming from that of $K$.  We then have   (see ~\cite{BT}) that the group 
of automorphisms of $G$ which we denote by $\Aut(G)$ is an extension of $\Aut_{\rm alg}(G)$ by $\Aut_G(K)$ i.e. the sequence 

$$1\rightarrow\Aut_{\rm alg}(G)\rightarrow\Aut(G)\rightarrow\Aut_G(K)\rightarrow 1$$
is an exact sequence. If $\bfg$ is a $K$-split algebraic group, then $\Aut(G)=\Aut_{\rm alg}(G)\rtimes\Aut(K)$, see~\cite[5.8, 5.9, 5.10]{Ti1} and references there.

From the description of $X_G$ given above, one sees that 
the group $\Aut(G)$ acts on the $X_G$ by simplicial automorphisms, 
giving a representation $$\rho: \Aut(G)\to\Aut(X_G).$$

The central theorem about automorphisms of buidlings is the following.

\begin{theorem}[Tits \cite{Ti1}]
\label{thm:tits}
Assume that ${\rm{rank}}_K\mathbf{G}>1$.  Then the representation

$$\rho: \Aut(G)\to\Aut(X_G)$$
is an isomorphism.
\end{theorem}

Note that $G$, which is subgroup of index at most $2$ in $\Aut_{\rm alg}(G),$ is a normal subgroup of $\Aut(X_G).$   The group $\Aut(X_G)$ is a locally compact group with respect to the compact-open topology.  This topology coincides with the topology on $\Aut(X_G)$ determined by the property that the sequences of neighborhoods about the identity map correspond to sets of automorphisms that are the identity on larger and larger balls in $X_G$.  
On the other hand, the groups $G$ and $\Aut_G(K)$ inherit a topology from the topology on $K$.  The isomorphism given in Theorem \ref{thm:tits} is an isomorphism of topological groups.


\subsection{Apartments and root subgroups}

The apartments (maximal flats) in $X_G$ correspond to maximal diagonalizable subgroups in $\tg.$ Suppose $S$ is a maximal diagonalizable subgroup of $\tg$, and let $A$ be the corresponding apartment in $X_G.$ Then $S$ acts on $A$ by translation. The root subgroups corresponding to $S$ acts on $X_G$ as follows.  Any root subgroup determines a family of parallel hyperplanes in $A$.  
If $u$ lies in the root subgroup it will fix a half-apartment of $A$, i.e.\  one component of the complement of some hyperplane $P$ in $A$. Moreover, $P$ is an intersection of apartments, and the action of the root group is transitive on the link of $P$ ( see \S 1.4 and \S 2.1 of \cite{Ti2} or, 
alternatively, Proposition 18.17 of \cite{W}). In particular we have the following.
 
\begin{fact}
\label{fact:flip}
Let $G$ be as above.  Then for any $(r-1)$-simplex $\sigma$ of $X_G$, and for any three $r$-simplices $\alpha_1,\alpha_2,\alpha_3$ having $\sigma$ as their common intersection, there exists an element $\phi\in G^+$ fixing $\alpha_1$ and switching $\alpha_2$ and $\alpha_3$.
\end{fact}

As an example consider $G=\PGL_2(\mathbf{Q}_p)$.  Then $X_G$ is a $(p+1)$-regular tree. Let $\ell$ be the apartment in $X_G$ corresponding to the diagonal group of $G$. In this case $\ell$ is 
a bi-infinite geodesic in $X_G$. Let $\ell(0)$ be the vertex corresponding to $\PGL_2(\zbf_p)$, 
i.e.\ the vertex corresponding to the standard lattice $\zbf_p^2.$  The geodesic ray 
$\ell([0,\infty))$ is a half-apartment based at $\ell(0)$. The above fact gives that there are elements $u_1,\ldots,u_{p-1}$ in one of the corresponding root groups (more precisely $u_i$'s are strictly upper triangular matrices) which map the ray $\ell((-\infty,0])$ to the other $(p-1)$-rays based at $\ell(0)$ intersecting $\ell$ only at that point.  These may be taken to be the representatives of nontrivial cosets in $\zbf_p/p\zbf_p$ if we identify the root group with the additive group $\mathbf{Q}_p.$

In this paper we will assume ${\rm{rank}}_K\bfg>1$.  Suppose $\Gamma$ is a lattice in $G$. Then the Margulis Superrigidity Theorem, proved in positive characteristic by Venkataramana~\cite{V} (see also \cite{Ma}), implies (by an argument of Margulis) that $\Gamma$ is superrigid and hence arithmetic.

\section{Proving extremal symmetry}

\label{sec;proof}

We begin by proving some lemmas and propositions that are used in the proof of both 
Theorem \ref{theorem:first} and Theorem \ref{theorem:second}.  

\subsection{Topological (non)rigidity of $X_G$}
\label{section:topological}

The following is a kind of topological rigidity result for $X_G$: it 
gives that the topological structure of $X_G$ remembers the simplicial structure. It is worth mentioning that in section~\ref{section:topological} we only need $Y$ to be locally compact simplicial complex homeomorphic to $X_G.$

\begin{proposition}\label{homeo-aut}
Let $f:X_G\rightarrow X_G$ be a homeomorphism. Then $f$ maps $k$-dimensional simplicies of $X_G$ onto $k$-dimensional simplicies for each $0\leq k\leq{\rm{dim}}(X_G).$ Hence there is a natural homomorphism $$\psi:{\rm{Homeo}}(X_G)\rightarrow\Aut(X_G).$$
\end{proposition}

\begin{proof}
We call a point $x\in X_G$ a {\em $k$-manifold point} of $X_G$ if a $x$ has some neighborhood homeomorphic to $\R^k$, and $k$ is the maximal such number so that this is true.  As mentioned above, $X_G$ is a thick building,  that is for each $k<\mbox{dim}(X_G),$ every $k$-dimensional simplex of $X_G$ is the face of at least three $(k+1)$-dimensional simplices of $X_G$.    
From this we clearly have the following:

\medskip
{\it Let $x\in X_G$ be any point.  Then $x$ is a $k$-manifold point if and only if $x$ lies in the interior of a $k$-simplex of $X_G$.}

\medskip
As being a $k$-manifold point is clearly a topological property for any fixed $k$, it follows that any homeomorphism $f:X_G\to X_G$ maps $k$-manifold points to themselves, and therefore $f$ 
maps open $k$-simplices into open $k$-simplices, for each $0\leq k\leq\dim(X_G)$.   
Applying the same argument to $f^{-1}$, we see that $f$ maps each open $k$-simplex 
of $X_G$ homeomorphically onto an open $k$-simplex of $X_G$.  

As $f$ is a homeomorphism it preserves adjacencies between simplices, and so $f$ induces a simplicial automorphism of $X_G$.  This association of $f$ to the simplicial automorphism it induces is clearly a homomorphism.
\end{proof}

Recall that $Y$ and $X_G$ are homeomorphic. Thus we have that $$\Aut(Y) \subseteq{\rm{Homeo}(Y)}\approx {\rm{Homeo}}(X_G).$$ We will denote by $\iota$ the restriction 
to $\Aut(Y)$ of the homomorphism $\psi$ defined in Proposition~\ref{homeo-aut}.

\begin{lemma}
Let $\Aut(Y)$ and $\Aut(X_G)$ be endowed with the compact-open topology.  Then the 
homomorphism $\iota:\Aut(Y)\to\Aut(X_G)$ is proper and continuous.
\end{lemma}

\begin{proof}
Continuity follows from the definitions.  To see that $\iota$ is proper, first note that $\Aut(Y)$ is locally compact since $Y$ is assumed to be locally finite. and that for any compact set $K$ in $\Aut(X_G)$ we have that $\iota^{-1}(K)$ is bounded. Now suppose that we are given any sequence $\varphi_n\in\Aut(Y)$ such that $\{\iota(\varphi_n)\}$ converges to $\tau\in\Aut(X_G)$.  Then $\{\varphi_n\}$ is pre-compact in 
$\Aut(Y)$, and for any limit point $\varphi_\infty$ we have that $\iota(\varphi_\infty)=\tau$.  Hence $\iota$ is a proper map.
\end{proof}

In contrast to rigidity, it is easy to see that the kernel of $\psi$ is huge.  Indeed it clearly contains 
the infinite product, over all maximal simplices $\sigma$, of the group of homeomorphisms 
of the closed $\dim(X_G)$-disk which are the identity on $\partial\sigma$.  On the other hand, when restricted to the subgroup $\Aut(Y)$, the map $\psi$ is injective.

\begin{proposition}\label{embedding}
The homomorphism $\iota:\Aut(Y)\rightarrow\Aut(X_G)$ is injective
\end{proposition}

\begin{proof}
Let $\varphi\in\mbox{ker}(\iota).$  We will argue inductively on the dimension $k\geq 0$ that $\varphi$ is the identity on the $k$-skeleton of $Y$.  Since $\iota(\varphi)=\id,$ we get $\varphi(v)=v$ for any vertex $v\in X_G$.  Now assume that $\varphi$ is identity on each $j$-simplex of $X_G$ for 
each $j<k$.   Let $D$ be any $k$-simplex of $X_G$.  Since $\iota(\varphi)=\id$, we have from the definition of $\iota$ that $\phi(D)\subseteq D$.  By induction we have that $\phi(x)=x$ for each $x\in \partial D$.  

Since $Y$ is a locally finite complex and $\varphi$ is simplicial automorphism of $Y$, we have that orbits of points under $\varphi$ are discrete.  Since $D$ is compact, there are only $m$ 
simplices of $Y$ intersecting $D$ for some $m<\infty$.  It follows that there exists $n$, depending only on $m$, so that $\varphi^n(x)=x$ for any $x\in D$.   Let $\tau:=\varphi|_D$.

Suppose that $\tau\neq\id$.  Then after raising $\tau$ to a power we can (and will) assume that $\tau$ has order $p$ for some prime $p$.   

Since we have a $p$-group $\langle\tau\rangle$ acting on a closed disk $D$, we can apply Smith  Theory to this action.   The pair $(D,\partial D)$ is of course a homology $k$-ball.  By Smith's Theorem (see, e.g.\ \cite{Br}, Theorem III.5.2), the pair $(\mbox{Fix}(\tau),\mbox{Fix}(\tau|_{\partial D})$ 
is a mod-$p$ homology $r$-ball for some $0\leq r\leq k$.    Since $\tau|_{\partial D}=\id$ by the induction hypothesis, we have that $\mbox{Fix}(\tau|_{\partial D})={\partial D}$, it follows that $r=k$.   

Now suppose that $\Fix(\tau)\neq D$.   Pick $x\in D$ in the complement of 
$\Fix(\tau)$.  Then radial projection away from $x$ to $\partial D$ gives 
a homotopy equivalence of pairs 
$$(\mbox{Fix}(\tau),\mbox{Fix}(\tau|_{\partial D})\simeq (\partial D,\partial D).$$
But this contradicts the fact that $(\mbox{Fix}(\tau),\mbox{Fix}(\tau|_{\partial D}))$ 
is a mod-$p$ homology $k$-disk with $k>0$, since as such, we have 
$$H_k(\mbox{Fix}(\tau),\mbox{Fix}(\tau|_{\partial D});\mathbf{Z}/p\mathbf{Z})
\neq 0=H_k(D,{\partial D};\mathbf{Z}/p\mathbf{Z}).$$

Thus it must be that $\Fix(\tau)=D$; that is, $\tau=\id$.  We have just proven that 
$\phi|_D=\id$ for each $k$-simplex $D$ of $X_G$, so by the 
induction on $k$ we have $\phi=\id$, as desired.
\end{proof}

\subsection{The extremal symmetry of $B_\Gamma$}

With the results from subsection \ref{section:topological} in hand, we can now prove Theorem \ref{theorem:first}.

\bigskip
\noindent
{\bf Proof of Theorem \ref{theorem:first}.}  
First note that any simplicial automorphism of $B_\Gamma$ induces an automorphism of $\pi_1(B_\Gamma)=\Gamma$, well-defined up to conjugacy.  We thus have a homomorphism 
$$\nu:\Aut(B_\Gamma)\to\Out(\Gamma)$$
where $\Out(\Gamma)$ is the group of {\em outer automorphisms} of $\Gamma$, i.e.\ the quotient of $\Aut(\Gamma)$ by inner automorphisms.  

We claim that $\nu$ is injective.  
 Suppose $f\in\ker(\nu)$.   Since $B_\Gamma$ is aspherical and $f_\ast$ acts trivially (up to conjugation) on $\pi_1(B_\Gamma)$ , it follows that $f$ is freely homotopic to the identity map.   
 Metrize $B_\Gamma$ so that it has the path metric induced by giving each simplex the standard Euclidean metric; $X_\Gamma$ then inherits a unique path metric making the covering $X_\Gamma\rightarrow B_\Gamma$ a local isometry.  
 
Since $f$ is homotopic to the identity and since $B_\Gamma$ is compact, each track in this homotopy moves points of $B_\Gamma$ some uniformly bounded distance $D$.  Thus $f$ has some lift $\widetilde{f}\in\Aut(X_G)$ such that $\widetilde{f}$ moves each point of $X_G$ at most a distance $D$.   We claim that the only element of $\Aut(X_G)$ that moves all points of $X_G$ at most a uniformly bounded distance is the identity automorphism.  Given this claim, it follows 
that $\widetilde{f}$, and hence $f$, is the identity, so that $\nu$ is injective.

The claim is well known, but for completeness we indicate a proof.  The building $X_G$ admits a nonpositively curved (in the ${\rm CAT}(0)$ sense) metric with the property that $\Aut(X_G)=\Isom(X_G)$.  Now, the boundary $\partial X_G$ of $X_G$ as a nonpositively curved space, namely the set of Hausdorff equivalence classes of infinite geodesic rays, can be identified with the 
spherical Tits building associated to $G$ (see~\cite[Theorem 8.24 and Chapter 28]{W}) . By the nonpositive curvature condition, infinite geodesic rays in $X_G$ either stay a uniformly bounded distance from each other, hence represent the same equivalence class in $\partial X_G$, or 
diverge with distance between point being unbounded.  If an element $\phi\in\Aut(X_G)$ moves all points of $X_G$ a uniformly bounded distance, it follows that $\phi$ induces the identity map on $\partial X_G$.  
But the natural homomorphism $\Aut(X_G)\to\Aut(\partial X_G)$ is injective (see~\cite{Ti1} or~\cite[Theorem 12.30 and Section 28.29]{W}), from which it follows that $\phi$ is the identity, proving the claim.

We now claim that $\nu$ is surjective, and thus is an isomorphism. To see this, note that by the assumptions on $G$, we can apply the 
Margulis Superrigidity Theorem, proved in positive characteristic by Venkataramana~\cite{V}, see~\cite{Ma}, to the lattice $\Gamma$ in $G$.  This gives 
in particular that $\Gamma$ satisfies {\em strong (Mostow-Prasad) rigidity}, which means that 
any automorphism of $\Gamma$ can be extended to a continuous homomorphism of $G.$  Note that the group of continuous automomorphisms of $G$ is precisely $\Aut(X_G).$  Thus given 
any $h\in\Out(\Gamma)$, there is some $h'\in \Aut(X_G)$ extending (a representative of) $h$, and so preserving $\Gamma$ in $G$.  Thus $h'$ descends to the desired automorphism of $B_\Gamma$, proving that $\nu$ is surjective.  We have thus shown that $\nu$ is an automorphism.

Now each $\varphi\in\Aut(\cg)$ induces an automorphism of $\pi_1(\cg)=\Gamma,$ which is well-defined up to conjugacy. Hence we obtain a homomorphism  $$\alpha:\Aut(\cg)\rightarrow\mbox{Out}(\Gamma).$$  

Since we just proved that $\nu^{-1}:\Out(\Gamma)\to\Aut(B_\Gamma)$ is an isomorphism, to prove the theorem it is enough to prove that $\alpha$ is injective.  
To this end, consider $\varphi\in\ker(\alpha)$.  Let $Y$ denote the universal cover of $C$.  Then, just as in the argument above,  $\varphi$ lifts to some $\widetilde{\varphi}$ moving points a bounded distance from the identity.  Here we are using the metric induced from the simplicial structure on $X_G$, not on $Y$ (although these metrics are uniformly comparable, so it doesn't actually matter).  By Proposition \ref{embedding} the homomorphism $\iota:\Aut(Y)\rightarrow\Aut(X_G)$ is one to one.  As was mentioned above the only element of $\Aut(X_G)$ moving points a uniformly bounded distance is the identity, we have $\widetilde{\varphi}=\id$, so that $\varphi=\id$ and $\alpha$ is injective, as desired. $\diamond$
\bigskip

\subsection{Characterizing $X_G$ among all simplicial structures}

The following result, crucial to our proof of Theorem \ref{theorem:second}, gives the consequence discussed at the end of the introduction.

\begin{proposition}[Coloring rigidity]
\label{discvscof}
Let the notation and assumptions be as above. Then 
precisely one of the following holds:
\begin{itemize}
\item[(i)] $\Aut(Y)$ is discrete.
\item[(ii)] $G^+\subseteq\iota(\Aut(Y))$, where $\iota$ is the monomorphism in Proposition~\ref{embedding}.
\end{itemize}
\end{proposition}

\begin{proof}
Recall that $\iota$ is a proper map. Hence $\iota(\Aut(Y))$ is a closed subgroup of $\Aut(X_G)$ with respect to the compact-open topology.  The continuity of $\iota$ together with Proposition~\ref{embedding} imply that if $\iota(\Aut(Y))$ is discrete then $\Aut(Y)$ is discrete and so (i) holds and we are done. 

We thus assume now that (i) does not hold.    So there is a sequence of elements $\varphi_n\in\Aut(Y)$ such that $\{g_n=\iota(\varphi_n)\}$ converges to the identity in $\Aut(X_G).$ 

Note that $\Gamma\subseteq\Aut(Y)$ and, with this abuse of notation, $\iota(\Gamma)=\Gamma.$ Note that $H:=G\cap\iota(\Aut(Y))$ is closed normal subgroup of $\iota(\Aut(Y))$ containing $\Gamma.$ We claim that $H$ is indiscrete. Assume the contrary and let $g_n$ be as above.  Since 
$g_n$ converge to identity it follows that $g_n\gamma g_n^{-1}\rightarrow\gamma$ for any $\gamma\in\Gamma.$ Since $H$ is normal and discrete, and since $\Gamma\subset H$, 
it follows that $g_n\gamma g_n^{-1}=\gamma$ for $n$ large enough. 

By the assumption $\mbox{rank}_KG\geq 2$, the group $G$ has Kazhdan's property T, and so 
the lattice $\Gamma$ in $G$ is finitely generated.   Hence there exists some $n_0$ such that if $n>n_0$ then $g_n\gamma g_n^{-1}=\gamma$ for all $\gamma\in\Gamma$ 

Thus for each $n\geq n_0$ we have that $g_n$ centralizes $\Gamma$. Such $g_n$ however is the trivial isometry. This follows, for example, from the proof of Theorem~\ref{theorem:first}. To be more explicit $g_n$ centralizes $\Gamma$ thus it induces the trivial isometry of $B_\Gamma,$ since $\Aut(B_\Gamma)$ and $\Out(\Gamma)$ are isomorphic, as we showed in loc. cit. Note now that $g_n$ centralizes $\Gamma$ so the action of $g_n$ on $X_G$ is trivial. Thus $g_n$ is identity if $n\geq n_0,$ which is a contradiction. Hence $H$ is indiscrete.

Recall that $G/\Gamma$ has a finite $G$-invariant measure and $\Gamma\subset H,$ hence $G/H$ has a finite $G$-invariant measure, namely the direct image of the measure on $G/\Gamma$ under the natural map $G/\Gamma\rightarrow G/H.$ Also we showed above that  
$H$ is an indiscrete subgroup of $G$. Now \cite[Chapter II, Theorem 5.1]{Ma} states that such a subgroup must contain $G^+$, as we wanted to show.
\end{proof}

We are now ready to prove the main theorem of this paper.

\bigskip
\noindent
{\bf Proof of Theorem \ref{theorem:second}. } There is nothing to prove if $\Aut(Y)$ is discrete, so 
suppose this is not the case. By Proposition \ref{embedding} and Proposition \ref{discvscof}, there exists a subgroup $H\subseteq\Aut(Y)$ such that $\iota:H\to G^+$ is an isomorphism.  It follows from the definition of $\iota$ that if $\varphi\in\Aut(Y)$ then  $\iota(\varphi)(v)=\varphi(v)$ for any vertex $v\in X_G$.  Actually, the proof of Proposition~\ref{homeo-aut} immediately gives: if $D$ is any simplex in $X_G$ then $\iota(\varphi)(D)=\varphi(D)$ is a simplex of $X_G$ whose verticies are the $\varphi$-images of the vertices of $D.$

We claim that for any simplex $\Dbf$ of $Y$, there is some chamber (simplex of maximal dimension) $C$ of $X_G$ such that $\Dbf\subseteq C.$  We prove this by induction on the dimension $k\geq 0$ of the cell $\Dbf.$ When $k=0$ this is trivial.  Now assume the claim is true up to dimension $k-1$.  

Let 
 $C(k)$ be the standard Euclidean $k$-dimensional simplex, and 
 $\beta: C(k)\rightarrow\Dbf$ be a simplicial parameterization. The induction hypothesis guarantees that any simplex of $\beta(\partial (C(k)))$ is contained some chamber of $X_G$.  Of course   
this chamber may not be unique.  If $\beta(C(k))$ is not contained in a single chamber,  
then there exists $x\in C(k)^\circ$, a neighborhood $B_\delta(x)\subseteq C(k)^\circ$, and two adjacent chambers $C_0$ and $C_1$ of $X_G$ such that $B_\delta(x)\cap C_i^\circ\neq\varnothing$ for each $i=0,1.$ Without loss of generality we can (and will) assume that $\beta(x)\in C_0\cap C_1.$

Since the building $X_G$ is thick,  there exists a chamber $C_2$ distinct from $C_0$ and $C_1$ such that the facet (i.e.\ codimension one face) $C_0\cap C_1$ is a facet of $C_2$ also.  
By Fact \ref{fact:flip} above, there is an element $u\in G^+$ such that $u|_{C_0}=\id$ and $u(C_1)=C_2.$ Let $\varphi\in H$ be such that $\iota(\varphi)=u.$ We have $\varphi(C_0)=C_0$ and $\varphi(C_1)=C_2.$ We also have $\varphi(v)=v$ for each vertex $v$ of $C_0.$ 

Now if we argue as in the proof of Proposition~\ref{embedding} we get $\varphi|_{C_0}=\id.$ This  implies that $\varphi|_{\beta(B_\delta(x)\cap C_0)}=\id$ and $\varphi(\beta(B_\delta(x)\cap C_1))\subseteq C_2.$ Recall however that $\varphi$ is a simplicial automorphism of $Y$, so that $\varphi(\Dbf)$ is another $k$-cell of $Y.$ Further, $\varphi(\beta(x))=\beta(x)$ is an interior point for two different $k$-cells of $Y$, namely $\Dbf$ and $\varphi(\Dbf)$. This is a contradiction.  Thus the claim that $\sigma$ lies in some chamber of $X_G$ follows.  

Let $C$ be a chamber which is a fundamental domain for the standard action of $G^+$ on 
$X_G$.   As a consequence of the claim above, we have that the restriction of the simplicial structure of $Y$ to each chamber of $X_G$ gives a simplicial  subdivision of the chamber. In particular $C$ is simplicially subdivided by $Y.$ Recall that since $\Aut(Y)$ is indiscrete, there is a subgroup $H$ which is isomorphically mapped to $G^+$ by $\iota.$ Now as $G^+$ acts transitively on chambers, we have that if $C'$ is any chamber of $X_G$ then there is some $\varphi\in H$ such that $\iota(\varphi)(C)=C'$.  By the remark we made in the beginning of the proof, we have $\varphi(C)=C'.$ The proof of the theorem is now complete. 
$\diamond$

\section{Explicit examples}
\label{sec;example}

In this section we give explicit examples of the arithmetic complexes to which Theorem \ref{theorem:first} and Theorem \ref{theorem:second} apply.  We then give examples in the rank one case where the loc. cit. do not apply. 

\bigskip
\noindent
{\bf An explicit example where Theorems~\ref{theorem:first} and~\ref{theorem:second} apply.} The explicit construction of these examples is given~\cite{LSV}, using lattices constructed in~\cite{CS}. These examples where constructed as explicit examples of ``{\em Ramanujan complexes}". Similar (explicit) constructions of complexes for which the above theorems holds are possible in characteristic zero using lattices constructed in \cite{CMSZ1}, \cite{CMSZ2} and \cite{MS}. 

Let $G=\PGL_3(\fbf_2((y)))$.  We want to describe a quotient of $X_G$ by a lattice $\Gamma$ which is a congruence subgroup of a lattice $\Gamma'$, where $\Gamma'$ acts simply transitively on the vertices of $X_G.$ Note that the building $X_G$ is in fact a {\em clique complex}; that is, 
any set of $k+1$ vertices is a cell if and only if every two vertices form a $1$-cell. This property holds for quotient complexes as well.  Thus, in order to describe the simplicial complex $B_\Gamma.$ it suffices to describe the Caley graph of $\Gamma'/\Gamma$ with an explicit set of generators.   
  
Let $t$ be a generator for the field of $16$ elements whose minimal polynomial is $t^4+t+1$.  
In other words, $\fbf_{16}=\fbf_2[t]/\langle t^4+t+1\rangle.$ The following set $S$ of seven matrices generates $\PGL_3(\fbf_{16}).$ The clique of the Caley graph of $\PGL_3(\fbf_{16})$ with respect to this set of generators is the complex obtained by taking the quotient of $X_G$ by a lattice $\Gamma,$ as above. This lattice is a congruence lattice of a lattice $\Gamma'$ which is constructed using a division algebra which splits at all places except at $1/y$ and $1/(y+1)$, at which it remains a division algebra.  

The set $S$ consists of the following seven matrices:

$$\left(\begin{array}{ccc}t+t^3& t^2& t+t^2\\ t& t^3& 1+t+t^2\\ t+t^2& 1+t^2& 1+t^3\end{array}\right)\hspace{4mm}\left(\begin{array}{ccc}1+t+t^2+t^3& t+t^2& 1+t^2\\ 1+t& t^2+t^3& 1\\ 1+t^2& t& t^3\end{array}\right)$$

$$\left(\begin{array}{ccc}1+t^2+t^3& 1+t^2& t\\ 1+t+t^2& t+t^3& t^2\\ t& 1+t& t^2+t^3\end{array}\right)\hspace{4mm}\left(\begin{array}{ccc}t+t^2+t^3& t& 1+t\\ 1& 1+t+t^2+t^3& t+t^2\\ 1+t& 1+t+t^2& t+t^3\end{array}\right)$$

$$\left(\begin{array}{ccc}1+t^3& 1+t& 1+t+t^2\\ t^2& 1+t^2+t^3& 1+t^2\\ 1+t+t^2& 1& 1+t+t^2+t^3\end{array}\right)\hspace{4mm}\left(\begin{array}{ccc}t^3& 1+t+t^2& 1\\ t+t^2& t+t^2+t^3& t\\ 1& t^2& 1+t^2+t^3\end{array}\right)$$

$$\left(\begin{array}{ccc}t^2+t^3& 1& t^2\\ 1+t^2& 1+t^3& 1+t\\ t^2& x+x^2& t+t^2+t^3\end{array}\right)$$

\bigskip
\noindent
{\bf An example in rank one case.} We begin with an example of an (arithmetic) lattice $\Lambda$ in $G=\PGL_2(\qbf_5)$,  given with a symmetric generating set of $\Lambda$ with $6$ elements, which acts simply transitively on $X_G.$ In other words $X_G,$ which is a $6$-regular tree, is the Caley graph of $\Lambda.$ This lattice $\Lambda$ is also used in~\cite{LPS} to construct explicit examples of ``Ramanujan graphs". Let 
$$\mathbf{H}(\zbf)=\{\alpha=a_0+a_1\mathbf{i}+a_2\mathbf{j}+a_3\mathbf{k}:\h a_i\in\zbf\}\hspace{3mm}$$ 
where $\mathbf{i}^2=\mathbf{j}^2=\mathbf{k}^2=-1$ and $\mathbf{i}\mathbf{j}=-\mathbf{j}\mathbf{i}=\mathbf{k}$.  For any $\alpha\in\mathbf{H}(\zbf)$ we let $\bar{\alpha}=a_0-a_1\mathbf{i}-a_2\mathbf{j}-a_3\mathbf{k}$ and let $N(\alpha)=\alpha\bar{\alpha}.$ Let $$\Lambda'=\{\alpha\in\mathbf{H}(\zbf): N(\alpha)=5^k,\h k\in\zbf\h\mbox{and}\h\alpha\equiv_21\}.$$ 

Now let $$\Lambda=\Lambda'/\sim$$ where 

$$\alpha\sim\beta \mbox{\ \ \  if\ \ \ } 5^{k_1}\alpha=\pm5^{k_2}\beta \mbox{\ \ \ for some $k_1,k_2\in\zbf.$}$$
 Note that $\Lambda$ is an (arithmetic) subgroup of $\PGL_2(\mathbf{Q}_5)$ and $[\alpha][\bar{\alpha}]=1.$ It is easy to see, and is shown in~\cite[Section 3]{LPS}, that $\Lambda$ is actually a free group in $\{\alpha_1,\alpha_2,\alpha_3\},$ where $N(\alpha_i)=5$ and $a_0>0$ for each $i=1,2,3.$
We identify $X_G$ with the Caley graph of $\Lambda$ with respect to the generating set $S=\{\alpha_1,\bar{\alpha_1},\alpha_2,\bar{\alpha_2},\alpha_3,\bar{\alpha_3}\}.$   

Now let $\Gamma$ be the kernel of the map $\Lambda\to \zbf/4\zbf$ given by $\alpha_i\mapsto i$ for $i=1,2,3.$ Then $B_\Gamma=X_G/\Gamma$ is the Cayley graph of $\zbf/4\zbf$ with respect to this generating set; that is, it is the complete graph with $4$ vertices. We now color the edges of $\bg$ with $3$ different colors so that the edges emanating from a vertex have $3$ different colors, and  we lift this to a coloring of $X_G$ using the $\Gamma$ action. 

Fix an arbitrarily large ball in $X_G$.  Consider the automorphism $\phi$ of the tree $X_G$ which fixes this ball pointwise and flips two rays corresponding to $\alpha_1$ and $\bar{\alpha_3}$ emanating from a vertex on the sphere and is identity everywhere else.  Then $\phi$ lies in the group of color-preserving automorphisms of this tree.  As the large ball was chosen arbitrarily, this argument proves that the group of color-preserving automorphisms of $X_G$ is not discrete.  
Of course we can replace different ``colors'' by different simplicial isomorphism types of triangulations of the corresponding simplices.  We thus have a contrast with the conclusion of Theorem~\ref{theorem:second}.



\bigskip
\noindent
Dept. of Mathematics\\
University of Chicago\\
5734 University Ave.\\
Chicago, IL 60637\\
E-mail: farb@math.uchicago.edu, amirmo@math.uchicago.edu
\medskip


\begin{thebibliography}{ZZZZ}

\bibitem[BT]{BT}A.~Borel, J.~Tits.
{\em Homomorphismes ``abstraits" de groupes algebriques simples}
Ann. of Math., Second Series, {\bf 97}, no. 3 ( 1973) 499-571

\bibitem[AB]{AB}
P. Abramenko and K. Brown, Buildings.  Theory and applications, 
{\em Graduate Texts in Mathematics}, 248. Springer, New York, 2008. 

\bibitem[Br]{Br}
G. Bredon, Introduction to Compact Transformatin Groups, Academic Press, 1972.


\bibitem[CMSZ1]{CMSZ1} D.~I.~Cartwright, A.~M.~Mantero, T.~Steger, A.~Zappa,
{\it Groups acting simply transitively on the vertices of a building of type $\tilde{A}_2$, I},
Geometriae Dedicata {\bf 47} (1993) 143-166.

\bibitem[CMSZ2]{CMSZ2} D.~I.~Cartwright, A.~M.~Mantero, T.~Steger, A.~Zappa,
{\it Groups acting simply transitively on the vertices of a building of type $\tilde{A}_2$, II},
Geometriae Dedicata {\bf 47} (1993) 143-166.

\bibitem[CS]{CS} D.~I.~Cartwright, T.~Steger,
{\it A family of $\tilde{A}_n$-groups,}
Israel Journal of Math. {\bf 103} (1998) 125-140.

\bibitem[FW1]{FW1} B.~Farb, S.~Weinberger,
{\it Hidden symmetries and arithmetic manifolds,}
Geometry, spectral theory, groups, and dynamics, 111--119, Contemp. Math., 387, Amer. Math. Soc., Providence, RI, 2005.

\bibitem[FW2]{FW2} B.~Farb, S.~Weinberger,
{\it Isometries, rigidity and universal covers.}
Ann. of Math. ({\bf 2}) 168 (2008), no. 3, 915--940.

\bibitem[LPS]{LPS}A.~Lubotzky, R.~Phillips, P.~Sarnak,
{\em Ramanujan Graphs,}
Combinatorica, {\bf 8} (3) (1988) 261-277. 

\bibitem[LSV]{LSV}A.~Lubotzky, B.~Samuels, U.~Vishne,
{\it Explicit construction of Ramanujan complexes of type ${\rm A}_d$},
Europ. J. of Combinatorics. {\bf 26} (2005) 965-993.


\bibitem[Ma]{Ma} G.~A.~Margulis,
Discrete subgroups of semisimple Lie groups,
Ergeb. Math. Grenzgeb. {\bf 17}, Springer, Berlin, 1991.

\bibitem[MS]{MS}A.~Mohammadi, A.~Salehi Golsefidy,
{Discrete vertex transitive actions on Bruhat-Tits building,}
Preprint.

\bibitem[PR]{PR}V.~Platanov, A.~Rapinchuk,
Algebaric groups and number theorey,
Academic Press, 1993.


\bibitem[Ti1]{Ti1}J.~Tits,
Building of Spherical type and finite B-N pairs,
Lecture Notes in Mathemtics {\bf 386}, Springer-Verlag, Berlin-New York 1974.

\bibitem[Ti2]{Ti2}
J. Tits, Reductive groups over local fields, in {\em Automorphic Forms,
Representations and $L$-Functions} (Corvallis, Ore., 1977), I, Proc.
Sympos. Pure Math. 33, Amer. Math. Soc., Providence, 1979, 29--69.


\bibitem[Ve]{V}T.~N.~Venkataramana,
{\it On superrigidity and arithmeticity of lattices in semisimple groups over local fields of arbitrary characteristic,}
Invent. Math.  \textbf{92}  (1988), 255-306.

\bibitem[We]{W}R.~Wiess,
The structure of affine buildings,
Ann. of Math. Studies, {\bf 168}, 2009.

\end{thebibliography}
\end{document}